\documentclass[12pt]{amsart}
\usepackage{amssymb,upref,enumerate}

\usepackage{amsmath,amssymb,amscd,amsthm,amsxtra}
\usepackage{latexsym}
\usepackage{graphics,epsfig,color}
\usepackage{soul}
\allowdisplaybreaks

\def\XXint#1#2#3{{\setbox0=\hbox{$#1{#2#3}{\int}$}
\vcenter{\hbox{$#2#3$}}\kern-.5\wd0}}

\newtheorem{theor}{Theorem}

\newtheorem{corollary}{Corollary}
\newtheorem*{theorA}{Theorem A}
\newtheorem*{theorB}{Theorem B}
\newtheorem*{theorC}{Theorem C}
\theoremstyle{remark}
\newtheorem{remark}{Remark}

\numberwithin{equation}{section}

\newcommand{\abs}[1]{\left\lvert#1\right\rvert}
\newcommand{\norm}[2]{{\left\| #1 \right\|}_{#2}}
\newcommand{\bey}{\begin{eqnarray*}}
\newcommand{\eey}{\end{eqnarray*}}

\newcommand{\rn}{{\mathbb R}^n}

\newcommand{\re}{\mathbb R}

\newcommand{\rr}{\mathbb R}

\newcommand{\nn}{\mathbb N}
\newcommand{\N}{\mathbb N}

\newcommand{\hn}{\mathbb H^n}

\newcommand{\p}{\partial}

\newcommand{\cg}{\mathbb{G}}

\newcommand{\Lcm}{\mathcal{L}}
\newcommand{\poly}{\mathcal{P}}
\newcommand{\liea}{\mathfrak{g}}

\begin{document}


\subjclass[2000]{Primary 26D10, 31B10; Secondary 46E35, 35R03}

\keywords{Sub-elliptic Poincar\'e inequalities, multilinear operators, stratified Lie groups, Carnot groups}

\address{Kabe Moen, Department of Mathematics,
University of Alabama, P.O.  Box 870350, Tuscaloosa, AL 35487,
USA.}  \email{kabe.moen@ua.edu}

\address{Virginia Naibo, Department of Mathematics, Kansas State University, 138 Cardwell Hall,
Manhattan, KS 66506, USA.} \email{vnaibo@math.ksu.edu}

\thanks{Partial NSF support under the following grants is acknowledged: DMS 1201504 (first author), DMS 1101327 (second author).}

\title[High-order multilinear Poincar\'e and Sobolev inequalities]{Higher-order multilinear Poincar\'e and Sobolev inequalities in Carnot groups}
\author{Kabe Moen \and Virginia Naibo}
\date{\today}

\begin{abstract}
The notions of higher-order weighted multilinear Poincar\'e and Sobolev inequalities in Carnot groups are introduced. As an application, weighted Leibniz-type rules in Campanato-Morrey spaces are established.
\end{abstract}

\maketitle

\bigskip

\section{Introduction}

The classical Poincar\'e inequality 
\begin{equation}\label{eq:classical}
\left(\int_B |h(x)-h_B|^q\,dx\right)^{1/q}\le C\, \left(\int_B |\nabla h(x)|^p\,dx\right)^{1/p},\quad h\in C^1(B),
\end{equation}
where $B$ is an Euclidean ball in $\rn$ of radius $r(B),$  $h_B=\frac{1}{|B|}\int_Bh(x)\,dx,$ $1< p< n$ and $q=\frac{np}{n-p},$ can be proved as a consequence of the representation formula 
\begin{equation}\label{eq:repclassical}
|h(x)-h_B|\lesssim \int_{B}\frac{|\nabla h(y)|}{|x-y|^{n-1}}\,dy, \quad x\in B,
\end{equation}
and the fact that the fractional integral of order 1  is a bounded operator from $L^p$ into $L^q.$ Representation formulas analogous to \eqref{eq:repclassical} were proved for H\"omander vector fields with balls associated to the corresponding Carnot-Carath\'eodory metric in $\rn$ in works that include \cite{Lu92, FLW95, FLW95b, FLW96,  CDG97, LW98b, HK00}. As a result,  weighted Poincar\'e inequalities of the type 
\begin{equation*}
\left(\frac{1}{u_B}\int_B|(h(x)-a(h,B))|^qu(x)\,dx\right)^{1/q}\le C\, r(B)\, \left(\frac{1}{v_B}\int_B |{\bf{X}} h(x)|^pv(x)\,dx\right)^{1/p}, 
\end{equation*}
follow, where  ${\bf{X}}=\{X_j\}_{j=1}^l$ is a collection of vector fields satisfying H\"omander's condition,  $B$ is a Carnot-Carath\'eodory metric ball in $\rn$ of radius $r(B),$ and $a(h,B)$ is the average of $h$ over $B$ with respect to Lebesgue measure or with respect to the measure $u(x)dx.$  The indices $p,$ $q$ and the weights $u,$ $v$ are related through certain conditions that extend the cases $1<p<Q,$ $q=\frac{Qp}{Q-p}$ (here $Q$ stands for homogeneous dimension) in the unweighted situation $u=v=1.$ In addition to the articles previouly cited, we refer the reader to the seminal works \cite{FKS82, CW85},  regarding weighted Poincar\'e inequalities in the Euclidean setting, and \cite{J86}, for unweighted Poincar\'e inequalities relative to H\"omander vector fields. Other related references include \cite{FGW94, MSal95, GN96,DGPb}. As it is well-known,
 Poincar\'e inequalities, along with Sobolev inequalities (replace $h_B$ by 0), play an important role in the study of local regularity of solutions to partial differential operators associated to vectors fields. An example of such operators is given by  $\sum_{i,j=1}^l X_i^*(a_{i,j}(x) X_j),$ where  $X_j^*$ is the adjoint of $X_j,$ $j=1,\cdots, l,$ and  $\{a_{i,j}\}$ is a symmetric matrix satisfying a suitable degenerate or non-degenerate ellipticity condition.

In the particular setting of Carnot groups, representation formulas in the spirit of  \eqref{eq:repclassical} with  higher order derivatives in the right hand side  and $h_B$ replaced by suitable polynomials have also been extensively studied; see for instance \cite{FS82,CSV92, Lu00, LW00, LW04}. More precisely,
given  a Carnot group $\cg$  of homogeneous dimension $Q$ and Carnot-Carath\'eodory metric $d$ with respect to a family of generators ${\bf{X}}$, it holds that  
\begin{equation*}
\abs{h(x)- P_k(B,h)(x)} \le C\, \int_B\abs{{\bf{X}}^k h(y)}\frac{d(x,y)^k}{\abs{B_d(x,d(x,y))}}\,dy,\quad x\in B,
\end{equation*}
where  $B$ is a $d$-ball, $h\in C^k(B)$, $0<k \le Q,$ $P_k(B,f)$ is a suitable polynomial of degree less than $k$  and  $C$ is independent of $h,$ $x$ and $B$ (consult notation in Section \ref{secc:carnotgroups}). In the Euclidean   case these inequalities read as
\begin{equation*}
\abs{h(x)- P_k(B,h)(x)} \le C\, \int_B\frac{\abs{{\bf{\nabla}}^k h(y)}}{|x-y|^{n-k}}\,dy,\quad x\in B.
\end{equation*}
Boundedness properties of potential operators of order $k$ lead then to higher order weighted Poincar\'e and Sobolev inequalities. 

\bigskip

Bilinear (and multilinear) Poincar\'e  inequalities such as
\begin{align}\label{order1}
\left(\int_B\left(|fg-f_B g_B|u\right)^q\,dx\right)^{1/q} &\lesssim \left(\int_B(|{\bf{X}}f|v_1)^{p_1}
\,dx\right)^{1/p_1}\left(\int_B\left(\abs{g}
v_2\right)^{p_2}\,dx\right)^{1/p_2} \nonumber \\
&+ \left(\int_B(|f|v_1)^{p_1}
\,dx\right)^{1/p_1}\left(\int_B\left(\abs{{\bf{X}}g}
v_2\right)^{p_2}\,dx\right)^{1/p_2},
\end{align}
where ${\bf{X}}$ is a collection of vector fields satisfying H\"ormander's condition, were introduced and studied in \cite{MMN10} in the context of Carnot-Carath\'eodory spaces.  Here $\frac{1}{2}<p\le q<\infty$ with  $\frac{1}{p}=\frac{1}{p_1}+\frac{1}{p_2},$ and the weights satisfy certain Muckenhoupt weights-type conditions involving these indices.
Such inequalities  provide a valid alternative to  inequalities of the type
\begin{equation}\label{Poinpq}
\inf\limits_{a \in \rr} \left( \int_B |(fg)(x) - a|^q \, dx \right)^{1/q} \lesssim \left(\int_B |\nabla(fg)(x)|^p \, dx\right)^{1/p},
\end{equation}
which fail when $0 < p < 1$. Higher-order versions of  \eqref{Poinpq}, for instance,
\begin{equation}\label{Poinpq2}
\inf\limits_{P(x)} \left( \int_B |(fg)(x) - P(x)|^q \, dx \right)^{1/q} \lesssim \left(\int_B |\Delta(fg)(x)|^p \, dx\right)^{1/p},
\end{equation}
where the supremum is taken over all polynomials $P(x)$ of degree less than two, also fail in general for $0<p<1$ (see Remark~\ref{counterexample}). By H\"older's inequality a natural substitute for \eqref{Poinpq2} is given by
\begin{align}\label{Poinpq3}
\inf\limits_{P(x)} \left( \int_B |(fg)(x) - P(x)|^q \, dx \right)^{1/q}& \lesssim
\left(\int_B |\Delta f(x)|^{p_1} \, dx\right)^{1/p_1}
\left(\int_B |g(x)|^{p_2} \, dx\right)^{1/p_2}\nonumber\\
&+\left(\int_B |\nabla f(x)|^{p_1} \, dx\right)^{1/p_1} \left(\int_B |\nabla g(x)|^{p_2} \, dx\right)^{1/p_2}\\
&+ \left(\int_B |f(x)|^{p_1} \, dx\right)^{1/p_1} \left(\int_B |\Delta g(x)|^{p_2} \, dx\right)^{1/p_2},\nonumber
\end{align}
 where $p_1$ and $p_2$ satisfy $\frac{1}{p}=\frac{1}{p_1}+\frac{1}{p_2}.$ As we will see, inequality \eqref{Poinpq3} and weighted versions of it are indeed true for $1/2<p<1.$ 

The aim of this work is to introduce higher-order weighted multilinear Poincar\'e and Sobolev inequalities  in the spirit of \eqref{Poinpq3} in the general context of  Carnot groups  that are  valid even  when $0<p<1.$

 \bigskip
 
This article consists of three additional sections. Section \ref{sec:results} starts with  the necessary background on  Carnot groups and then presents   the statements of the main results of this article  along with some remarks and examples. In Section \ref{secc:proofmain} we  prove our main results, Theorem \ref{high2wthm}, Theorem~\ref{sobolev} and Theorem~\ref{sobolev2}, which follow from two key pieces in the context of Carnot groups: the boundedness of multilinear potential operators (Theorem \ref{multifracint}) and representation formulas for products of functions (Corollaries~\ref{repformulamulti}, \ref{globalrepformulamulti}, \ref{repformulamulti3}).   In Section \ref{secc:leib} we close this article with an application   to weighted Leibniz-type rules for Campanato-Morrey spaces.

\bigskip

{\bf Acknowledgements.}  The authors would like to thank Guozhen Lu for useful discussions regarding the results presented in this article. They also thank the anonymous referees for their comments and suggestions.

\section{Main Results}\label{sec:results}

We start this section by reviewing some preliminaries concerning Carnot groups (Section \ref{secc:carnotgroups}). The main results are presented in Section \ref{secc:mainres} along with some remarks and examples.

\subsection{Carnot groups}\label{secc:carnotgroups}  

We follow the notation in \cite{BLU07} for our exposition and refer the reader to \cite{BLU07, CSV92,FS82} for further details. 

A smooth vector field $X$ on $\rn$ is a $C^\infty$ function $X:\rn\to \rn,$ this is, $X(x)=(a_1(x),\cdots,a_n(x))^T,$ $x\in\rn,$ where $a_i:\rn\to\re,$ $i=1,\cdots,n,$ are infinitely differentiable functions. If $f:\rn\to\re$ is  differentiable, we denote by $Xf$ the function defined by
\[
Xf(x)=X(x)^T\cdot \nabla f(x)=\sum_{j=1}^na_j(x)\,\partial_jf(x),\qquad x\in \rn.
\]

  If ${\bf{X}}=\{X_1,\cdots,X_l\}$ is a family of smooth vector fields in $\rn,$ $f\in C^1(\rn),$ and  $\alpha=(\alpha_1,\cdots,\alpha_l)\in\nn_0^l$ is a multi-index  we define
\[
{\bf{X}}^{\alpha}f:=X_1^{\alpha_1}(\cdots X_{l-1}^{\alpha_{l-1}}(X_{l}^{\alpha_l}f)),
\]
and if $k\in\nn_0$ we set
\[
\abs{{\bf{X}}^kf}:=\left(\sum_{\abs{\alpha}=k} \abs{{\bf{X}}^\alpha f}^2\right)^{1/2}.
\]


Let $(\rn,\diamond)$ be  a Lie group on $\rn$ and denote by $\liea$ its Lie algebra. Consider  $n_1,\cdots,n_s\in \nn,$ $n_1+\cdots+n_s=n,$ and dilations $\{\delta_\lambda\}_{\lambda>0}$ of the form
\[\delta_\lambda(x)=(\lambda \,x^{(1)}, \lambda^2 x^{(2)},\cdots,\lambda^s x^{(s)}),\quad  x^{(i)}\in\re^{n_i}.\]
The triple $\cg=(\rn,\diamond,\delta_\lambda )$ is said to be a  Carnot group (of step $s$ and $n_1$ generators) if  $\delta_\lambda$ is an automorphism of $(\re^n, \diamond)$ for every $\lambda>0$ and if the first $n_1$ elements of the Jacobian basis of $\liea,$ say $Z_1,\cdots,Z_{n_1},$  satisfy
\begin{equation}\label{hormander}
\text{rank}(\text{Lie}[Z_1,\cdots,Z_{n_1}](x))=n,\qquad \text{for all }x\in\rn,
\end{equation}
where $\text{Lie}[Z_1,\cdots,Z_{n_1}]$ is the Lie algebra generated by the vector fields $Z_1,\cdots Z_{n_1}.$

The number $Q=\sum_{i=1}^s i\,n_i$ is called the homogeneous dimension of $\cg.$ The vector fields $Z_1,\cdots,Z_{n_1}$ are called the (Jacobian) generators of $\cg,$ whereas any basis for $\text{span}\{Z_1,\cdots,Z_{n_1}\}$ is called a system of generators of $\cg.$
It  follows that 
\[
\liea=W^{(1)}\oplus \cdots \oplus W^{(s)},
\]
where $W^{(i)}$  denotes the vector space spanned by the commutators of length $i$ of the vectors $Z_1,\cdots, Z_{n_1}.$ The elements in $W^{(i)}$ are $\delta_\lambda$-homogeneous of degree $i,$ and $\dim{W^{(i)}}=n_i$ if $i\le s$ and $W^{(i)}=\{0\}$
if $i\ge s.$ Note that this implies that  $\text{Lie}[Z_1,\cdots,Z_{n_1}]=\liea;$ moreover condition \eqref{hormander} implies that $\{Z_1,\cdots Z_{n_1}\}$ is a family of vector fields satisfying H\"ormander's condition.

We will consider in $\rn$ the Carnot-Carath\'eory metric $d$ associated to a system of generators of $\cg.$ If  $B_d(x,r)$ is the $d$-ball of radius $r$ centered at $x$ then $\abs{B_d(x,r)}=c_d \,r^Q$ where $c_d=\abs{B_d(0,1)}$ (see \cite[p. 248]{BLU07}).
It follows that  $(\rn,d,\text{Lebesgue measure})$ is a space of homogeneous type. 

 The second order differential operator
\[
\mathcal{L}=\sum_{j=1}^{n_1} X_i^2
\]
is called the canonical sub-Laplacian of $\cg$ if $X_i=Z_i,$ $i=1,\cdots,n_1,$ and simply a sub-Laplacian if $\{X_1,\cdots,X_{n_1}\}$ is a system of generators of $\cg.$ We point out that
there are characterizations of  families of smooth vector fields  $\{X_1,\cdots, X_{n_1}\}$ for which there exists a  Carnot group with respect to which $\sum_{i=1}^{n_1} X_i^2$ is a sub-Laplacian (see, for instance, \cite[p.191]{BLU07}).

If $\alpha=(\alpha_1,\cdots,\alpha_n)\in\nn_0^n$ is a multi-index and $x=(x_1,\cdots,x_n)\in\re^n$ we set 
\[x^\alpha=x_1^{\alpha_1}\cdots x_n^{\alpha_n},\quad \abs{\alpha}=\alpha_1+\alpha_2+\cdots+\alpha_n,\quad \abs{\alpha}_\cg=\sigma_1\alpha_1+\sigma_2\alpha_2+\cdots+\sigma_n\alpha_n,\]
where $\sigma_i=1$ for  $i=1,\cdots,n_1,$ $\sigma_i=2$ for $i=n_1+1,\cdots,n_1+n_2,$  $\sigma_i=3$ for $i=n_1+n_2+1,\cdots, n_1+n_2+n_3,$ and so on.
If $P(x)=\sum_{\alpha} c_\alpha x^\alpha$ is a polynomial on $\cg$ the homogeneous degree (or just degree) of $P$ is defined  as
$\text{deg}_\cg(P)=\max\{\abs{\alpha}_\cg: c_\alpha\neq 0\}.$

\subsection{Main results}\label{secc:mainres} We are now ready to present the  main results in this article.

\begin{theor}[Higher-order weighted multilinear Poincar\'e inequality]\label{high2wthm} Suppose $m \in \N,$ $\frac{1}{m}<p\le q<\infty$ and $1<p_1, \cdots,
 p_m <\infty$ such that
 $ \frac{1}{p}=\frac{1}{p_1}+\cdots+\frac{1}{p_m}.$
Let $\cg$ be a  Carnot group in $\rn$ of homogeneous dimension $Q$ and $n_1$ generators, $k$ and $m$  positive integers such that $k\le m\, Q$, $d$ the Carnot-Carath\'eodory metric in $\rn$ with respect to a family of generators ${\bf{X}}$ of $\cg.$
Let  $u, \, v_i,$ $i=1,\cdots, m,$  be
weights defined on $\re^n$
and satisfying condition \eqref{2wq>1} if $q>1$ or condition
\eqref{2wq<1} if $q\le 1,$ where

\begin{equation}\label{2wq>1}
\sup_{B \, d\text{-ball}}
r(B)^{k+Q(1/q - 1/p)} \left( \frac{1}{r(B)^Q}\int_B
u^{qt} dx \right)^{1/q t} \prod_{i=1}^m \left(
\frac{1}{r(B)^Q}\int_B v_i^{-t p_i'} dx \right)^{1/t p_i'} <
\infty,
\end{equation}
for some $t> 1,$

\medskip

\begin{equation}\label{2wq<1}
\sup_{B\, d\text{-ball}}
r(B)^{k+Q(1/q - 1/p)} \left( \frac{1}{r(B)^Q}\int_B
u^{q} dx \right)^{1/q} \prod_{i=1}^m \left( \frac{1}{r(B)^Q}\int_B
v_i^{-t p_i'} dx \right)^{1/t p_i'} < \infty,
\end{equation}
for some $t > 1$, with $r(B)$ the radius  of $B$.

Then  for all $d$-ball $B$ and all $\vec{f}=(f_1,\cdots,f_m)\in (C^k(\overline{B}))^m,$ there exists a polynomial $P_k(B,\vec{f})$ of degree less than $k$ such that the following weighted $m$-linear subelliptic Poincar\'e inequality holds true
\begin{align}\label{high2w}
&\left(\int_B\left(\abs{\prod_{i=1}^m f_i-P_k(B,\vec{f})}u\right)^q\,dx\right)^{1/q}&\le C\,
\sum_{\substack{\alpha_i\in \N_0^{n_1} \\ \abs{\alpha_1}+\cdots+\abs{\alpha_m}=k}}
\prod_{i=1}^m\left(\int_B\left(\abs{{\bf{X}}^{\alpha_i}f_i}
v_i\right)^{p_i}\,dx\right)^{1/p_i},
\end{align}
where $C$ is a constant independent of $\vec{f}$ and $B$  and $\bf{X}^\alpha$ for a multiindex $\alpha$ is as defined in Section \ref{secc:carnotgroups}.
\end{theor}

In the linear case $(m=1)$, representation formulas and Poincar\'e inequalities imply embedding theorems on Campanato-Morrey spaces.  These embeddings are applicable when studying the regularity of solutions to  partial differential equations; see, for instance, Lu \cite{Lu95, Lu98} where such embeddings were proven in the Carnot-Carath\'eodory context. The multilinear analogs of these embeddings  come in the form of Leibniz-type rules. We next illustrate this by focusing on the bilinear case $m=2$.
Indeed, the fractional Leibniz rule states that for $\alpha > 0$ and $1 < p_1, p_2,  q_1, q_2, r <
\infty,$ with $\frac{1}{r} = \frac{1}{p_1} + \frac{1}{q_1} = \frac{1}{p_2} + \frac{1}{q_2}$, the inequality
\begin{equation}\label{fracleibniz}
\|\abs{\nabla}^\alpha (fg)\|_{L^r}\lesssim \|\abs{\nabla}^\alpha f \|_{L^{p_1}}\|g\|_{L^{q_1}}+\|f\|_{L^{p_2}}\|\abs{\nabla}^\alpha g\|_{L^{q_2}}
\end{equation}
holds true, where $\widehat{\abs{\nabla}^\alpha h}(\xi)=\abs{\xi}^\alpha \hat{h}(\xi)$. Since the pioneering work by Christ-Weinstein \cite{CWein91}
and Kenig-Ponce-Vega \cite{KPVe93} on the Korteweg-de Vries equation and Kato-Ponce \cite{KaPo88} on the
Navier-Stokes equation, inequality \eqref{fracleibniz} has emerged
as an essential tool to study nonlinear PDEs.  In particular, PDEs whose nonlinear terms involve quadratic expressions,
or, more generally, powers of the solution or its derivatives, and products of the
solutions and their derivatives. The use of \eqref{fracleibniz} and closely related inequalities has vastly
spread across the literature in Analysis and PDEs. Notice that inequality
\eqref{fracleibniz} is a particular case of
\begin{equation}\label{fracleibniz2}
\|fg\|_{{\mathcal{Z}}} \lesssim \|f\|_{{\mathcal{X}}_1} \|g\|_{{\mathcal{Y}}_1}
+ \|f\|_{{\mathcal{X}}_2} \|g\|_{{\mathcal{Y}}_2}
\end{equation}
where the spaces ${\mathcal{Z}}= \dot{W}^{\alpha,r},$ ${\mathcal{X}}_1=\dot{W}^{\alpha,p_1}$, ${\mathcal{Y}}_1=L^{q_1}$ , ${\mathcal{X}}_2 = L^{p_2}$ , and ${\mathcal{Y}}_2=\dot{W}^{\alpha,q_2}$
belong to the
homogeneous Sobolev scale. As an application of Theorem~\ref{high2wthm} we derive inequalities of type \eqref{fracleibniz2} in the scale of weighted
 Campanato-Morrey spaces in Carnot groups (see Theorem~\ref{LeibnizCampanato}). We remark that these estimates are new even in the Euclidean setting.

\bigskip

When $k=1$ in Theorem~\ref{high2wthm} one has $P(B,\vec{f})\equiv\prod_{i=1}^m {f_i}_{B}$. Since constants on the right hand side of \eqref{high2w} are independent of $B$ , by taking $\abs{B}\to \infty,$ we obtain the following first-order weighted  Sobolev inequality
\begin{align*}
&\left(\int_{\rn}\left(\prod_{i=1}^m \abs{f_i}u\right)^q\,dx\right)^{1/q}\\&\le C\,
\sum_{i=1}^m\left(\int_{\rn}\left(\abs{{\bf{X}} f_i}
v_i\right)^{p_i}\,dx\right)^{1/p_i}\prod_{j\neq
i}\left(\int_{\rn}\left(\abs{f_j}
v_j\right)^{p_j}\,dx\right)^{1/p_j}, \nonumber
\end{align*}
for $f_i\in C^k_c(\re^n).$
More generally, as proved in  \cite{LW04} (see also \cite{Lu00}), the polynomials in Theorem~\ref{high2wthm} can be taken so that a limit argument gives the following higher-order weighted multilinear Sobolev inequalities.

\begin{theor}[Higher-order weighted multilinear Sobolev inequalities]\label{sobolev}
Under the same hypothesis as Theorem~\ref{high2wthm} and for $f_i\in C^k_c(\re^n)$, $i=1,\cdots,m,$
\begin{align}\label{high2wSobolev}
&\left(\int_{\rn}\left(\prod_{i=1}^m \abs{f_i}u\right)^q\,dx\right)^{1/q}&\le C\,
\sum_{\substack{\alpha_i\in \N_0^{n_1} \\ \abs{\alpha_1}+\cdots+\abs{\alpha_m}=k}}
\prod_{i=1}^m\left(\int_{\re^n}\left(\abs{{\bf{X}}^{\alpha_i}f_i}
v_i\right)^{p_i}\,dx\right)^{1/p_i},
\end{align}
where $C$ is a constant independent of $\vec{f}$ and $\bf{X}^\alpha$ for a multiindex $\alpha$ is as defined in Section \ref{secc:carnotgroups}.
\end{theor}

When $k=2,$ a representation formula in terms of the fundamental solution of a sub-Laplacian on a  Carnot group and the boundedness properties of the multilinear fractional integrals give the following  Sobolev inequality.

\begin{theor}[Second-order weighted multilinear Sobolev inequalities with sub-Laplacians]\label{sobolev2}
Under the same hypothesis as Theorem~\ref{high2wthm} (k=2) and for $f_i\in C^2_c(\re^n)$, $i=1,\cdots,m,$
\begin{align*}
&\left(\int_{\rn}\left(\prod_{i=1}^m \abs{f_i}u\right)^q\,dx\right)^{1/q}\\&\le C\,
\sum_{i=1}^m\left(\int_{\rn}\left(\abs{\mathcal{L} f_i}
v_i\right)^{p_i}\,dx\right)^{1/p_i}\prod_{j\neq
i}\left(\int_{\rn}\left(\abs{f_j}
v_j\right)^{p_j}\,dx\right)^{1/p_j}, \nonumber
\end{align*}
where  $\mathcal{L}$ is the sub-Laplacian associated to ${\bf{X}}$ and $C$ is a constant independent of $\vec{f}.$
\end{theor}

\begin{remark}We point out that Theorem \ref{high2wthm}, as well as the notions of higher-order weighted multilinear Poincar\'e and Sobolev inequalities, are new even in the Euclidean setting. On the other hand,  the Euclidean case of Theorem~\ref{sobolev} for $k=1$ and Theorem~\ref{sobolev2} were proved in \cite{Moen09}.
\end{remark}
\begin{remark} The multilinear techniques implemented in the proof of Theorem \ref{high2wthm} allow for a class of weights strictly larger than the one obtained by iteration of the linear results and H\"older-type inequalities, see  \cite[p. 10]{MMN10} and \cite[Remark 7.5]{Moen09}.
\end{remark}
\begin{remark}\label{counterexample} We present here an example that shows that inequalities of the type \eqref{Poinpq2} fail in general for $0<p<1.$ The ideas are  inspired by the example given in \cite{BK94} that proves that \eqref{Poinpq} may fail for $0<p<1.$

We first consider the one dimensional Euclidean case. Let $\varphi:[-1, 1]\to [0,1]$ be a continuously differentiable function such that $\varphi(-1)=0,$ $\varphi(1)=1,$ $\varphi'(-1)=\varphi'(1)=0$ and $\int_{-1}^1\varphi(x)\,dx=1.$ Set
\[
\psi(x):=\left\{\begin{array}{ll}
  0 & \text{ for } x\le -1,  \\
  \int_{-1}^x\varphi(t)\,dt & \text{ for } -1\le x\le 1, \\
  x & \text{ for } x\ge 1, 
\end{array}\right.
\]
 For each $\varepsilon\in (0,\frac{1}{2}),$ define $f_\varepsilon(x)=\varepsilon \,\psi(\frac{x}{\varepsilon})$ for $x\in\re.$
It follows that 
\[
\lim_{\varepsilon\to 0}\int_{-1}^1\abs{f_\varepsilon''(x)}^p\,dx=0, \qquad 0<p<1.
\]
On the other hand, if $0< q<\infty$ there exists a constant $C_q>0$ such that 
\[
\inf_{a,\,b\in\re} \int_{-1}^1\abs{f_\varepsilon(x)-(ax+b)}^q\,dx\ge C_q, \qquad \text{ for  }  \varepsilon \in (0,1/2).
\]
Indeed, if $0< q<\infty$ and $0<\varepsilon<1/2$ then
\begin{align*}
\inf_{a,\,b\in\re} \int_{-1}^1\abs{f_\varepsilon(x)-(ax+b)}^q\,dx & \ge C\, \inf_{a,\,b\in\re} \int_{1/2}^1\left(\abs{ax-b}+\abs{(1-a)x-b}\right)^q\,dx\\
& \ge C\, \inf_{b\in\re}\int_{1/2}^1\abs{x-2b}^q\, dx.
\end{align*}
Elementary computations show that $\int_{1/2}^1\abs{x-2b}^q\, dx$ is bounded from below by a positive constant independent of $b.$
For higher dimensions consider $g_\varepsilon(x_1,\cdots,x_n):=f_\varepsilon(x_1)$ where $f_\varepsilon$ is as above and integrate, say, in a ball centered at the origin and radius one.

\end{remark}

We conclude this section with  some examples of the result of  Theorem \ref{sobolev2} in the setting of  the Heisenberg group $\hn$.   Recall that $\hn=(\rr^{2n+1},\diamond, \delta_\lambda)$ where 
\[(x,y,s)\diamond(x',y',s')=(x+x',y+y',s+s'+2\,y\cdot x'- 2\, x\cdot y'),\quad x,y,x',y'\in\rn,\,s,s'\in\re,\]
 and 
\[
\delta_{\lambda}:\re^{2n+1}\to \re^{2n+1},\quad \delta_\lambda(x,y,s)=(\lambda x, \lambda y, \lambda^2 s).
\]
The homogeneous dimension of $\hn$ is then $Q=2n+2$ and  the Jacobian generators are 
$$X_i=\p_{x_i}+2y_i\p_s, \quad Y_i=\p_{y_i}-2x_i\p_s, \qquad i=1,\ldots, n.$$
The canonical    sub-Laplacian on $\hn$ (Kohn Laplacian) is given by 
\begin{align*}
\Delta_{\hn}=\sum_{j=1}^n X_j^2+Y_j^2
=\sum_{j=1}^n\partial_{x_j}^2+\partial_{y_j}^2+4\left(\sum_{j=1}^nx_j^2+y_j^2\right)\partial_s^2+4\left(\sum_{j=1}^ny_j \partial_{x_j}-x_j\partial_{y_j}\right)\partial_s.
\end{align*}
Set $z=(x,y,s)\in \hn$ and let
$$|z|_{\hn}=|(x,y,s)|_{\hn}=\sqrt[4]{(|x|^2+|y|^2)^2+s^2}.$$
The weights $u(z)=|z|_{\hn}^{-2}$, $v_1=v_2=1$ satisfy the condition \eqref{2wq<1} with $p=q=1$ and $p_1=p_2=2$ and any $t>1$.   
By Theorem~\ref{sobolev2} it follows that
\begin{multline*} \int_{\re^{2n+1}}\frac{|f(z)g(z)|}{|z|_{\hn}^2}\,dz \lesssim \left(\int_{\re^{2n+1}} |\Delta_{\hn} f(z)|^2\,dz\right)^{1/2} \left(\int_{\re^{2n+1}} |g(z)|^2\,dz\right)^{1/2} \\
+\left(\int_{\re^{2n+1}} |f(z)|^2\,dz\right)^{1/2} \left(\int_{\re^{2n+1}} |\Delta_{\hn} g(z)|^2\,dz\right)^{1/2}.
\end{multline*}
Another example with $p=q=1$ and $p_1=p_2=2$ is given by taking  $u$ to be an $A_\infty$ weight  in $\hn$ and $v_1(x)^2=v_2(x)^2=u(x)^{-1}$.  Condition \eqref{2wq<1} is then satisfied  (with $t>1$ coming from the reversed H\"older inequality satisfied by $u$) if
\begin{equation}\label{weightcond} \sup_{B {\rm{ \, ball\, in\, }} \hn }\frac{1}{r(B)^{Q-1}}\int_B u(z)\,dz<\infty.\end{equation}
For such $u,$ Theorem \ref{sobolev2} gives
\begin{multline*} \int_{\re^{2n+1}}|f(z)g(z)|u(z)\,dz \lesssim \left(\int_{\re^{2n+1}} \frac{|\Delta_{\hn} f(z)|^2}{u(z)}\,dz\right)^{1/2} \left(\int_{\re^{2n+1}} \frac{|g(z)|^2}{u(z)}\,dz\right)^{1/2} \\
+\left(\int_{\re^{2n+1}} \frac{|f(z)|^2}{u(z)}\,dz\right)^{1/2} \left(\int_{\re^{2n+1}} \frac{|\Delta_{\hn} g(z)|^2}{u(z)}\,dz\right)^{1/2}.
\end{multline*}
An example of a weight in $A_\infty$ that satisfies \eqref{weightcond} is $u(z)=|z|_{\hn}^{-1},$   hence, the following holds:
\begin{multline*} \int_{\re^{2n+1}}\frac{|f(z)g(z)|}{|z|_{\hn}}\,dz \lesssim \left(\int_{\re^{2n+1}} |\Delta_{\hn} f(z)|^2|z|_{\hn}\,dz\right)^{1/2} \left(\int_{\re^{2n+1}} |g(z)|^2|z|_{\hn}\,dz\right)^{1/2} \\
+\left(\int_{\re^{2n+1}} |f(z)|^2|z|_{\hn}\,dz\right)^{1/2} \left(\int_{\re^{2n+1}} |\Delta_{\hn} g(z)|^2|z|_{\hn}\,dz\right)^{1/2}.
\end{multline*}

\section{Proofs of Theorem~\ref{high2wthm}, Theorem~\ref{sobolev}, and Theorem~\ref{sobolev2}}\label{secc:proofmain}

The proofs of Theorems~\ref{high2wthm}, \ref{sobolev}, and \ref{sobolev2}  follow  from Theorem~\ref{multifracint} and  Corollaries~\ref{repformulamulti}, \ref{globalrepformulamulti} and \ref{repformulamulti3}, respectively, which we state and prove below.

\subsection{Multilinear fractional integrals in  Carnot groups.} Given a Carnot group $\cg=(\rn, \diamond, \delta_\lambda)$ of homogeneous dimension $Q$ and a system of generators of $\cg,$ say ${\bf{X}}=\{X_1,\cdots,X_{n_1}\},$ we consider on $\rn$ the Carnot-Carath\'eodory metric $d$ associated to ${\bf{X}}.$ If  $B_d(x,r)$ is the $d$-ball of radius $r$ centered at $x,$ recall that  $\abs{B_d(x,r)}=c_d \,r^Q$ where $c_d=\abs{B_d(0,1)}$ and therefore
 $(\rn,d,\text{Lebesgue measure})$ is a space of homogeneous type. For $\vec{x}=(x_1,\cdots,x_m)$ and $\vec{y}=(y_1,\cdots,y_m)$ with $x_i,\,y_i\in\re^n,$ $i=1,\cdots,m,$ we define $d(\vec{x},\vec{y}):=d(x_1,y_1)+\cdots+d(x_m,y_m);$ in the case when $x_1=x_2=\cdots=x_m=:x,$ we simply write $d(x,\vec{y})$ instead of $d(\vec{x},\vec{y}.)$

 In the framework of Carnot groups, we define the multilinear fractional integral of order  $\tau>0$  by
\begin{equation}\label{potopdefinition}
\mathcal{I}_{\cg,\tau} (\vec{f})(x):= \int_{{\re}^{nm}}
\frac{\vec{f}(\vec{y})}{d(x,\vec{y})^{m\,Q-\tau}}\,d\vec{y},\qquad x\in\rn,
\end{equation}
where $\vec{f}=(f_1,\cdots,f_m),$ $f_i:\re^n\to\re,$ and  for $\vec{y}=(y_1,\cdots,y_m)\in \re^{nm},$  $\vec{f}(\vec{y})=\prod_{i=1}^m f_i(y_i).$ 

  \begin{theor}\label{multifracint}
Suppose $m \in \N,$ $\frac{1}{m}<p\le q<\infty$ and $1<p_1, \cdots,
 p_m <\infty$ such that
 $ \frac{1}{p}=\frac{1}{p_1}+\cdots+\frac{1}{p_m}.$
Let $\cg$ be a  Carnot group in $\rn$ of homogeneous dimension $Q,$ $\tau$ a positive real number,  $m$  a positive integer such that $\tau\le m \,Q$, $d$ the Carnot-Carath\'eodory metric in $\rn$ with respect to a family of generators ${\bf{X}}$ of $\cg,$ and 
  $u, \, v_i,$ $i=1,\cdots, m,$ 
weights defined on $\re^n$
and satisfying condition \eqref{2wq>1} if $q>1$ or condition
\eqref{2wq<1} if $q\le 1$ with $k$ replaced by $\tau.$
Then there exists a constant $C$  such that
\[
\left(\int_{\rn}\left(\abs{\mathcal{I}_{\cg,\tau}(\vec{f})(x)}u(x)\right)^q\,dx\right)^{1/q}\le
C \prod_{i=1}^m\left(\int_{\rn}(\abs{f_i(x)}v_i(x))^{p_i}\,dx\right)^{1/p_i}
\]
for all $\vec{f}=(f_1,\cdots,f_m)\in L^{p_1}(\rn,v_1^{p_1}dx)\times\cdots\times
L^{p_m}(\rn,v_m^{p_m}dx).$ The constant $C$ depends only on structural constants and the constants appearing in  \eqref{2wq>1} and \eqref{2wq<1}.
\end{theor}

\begin{proof} Since $(\cg,d,\text{Lebesgue measure})$ is a space of homogeneous type  we just need to check that the hypothesis of \cite[Corollary 1]{MMN10}  are satisfied if $\tau\le m\,Q.$ The conditions to be checked are the reverse doubling property of Lebesgue measure with respect to $d$-balls and a growth condition for the kernel of the multilinear fractional integral.

The reverse doubling condition in this setting  means that there are positive  constants $c$ and $\delta$ such that
\[
\frac{\abs{B_d(x_1,r_1)}}{\abs{B_d(x_2,r_2)}}\ge c \left(\frac{r_1}{r_2}\right)^\delta,
\]
whenever $B_d(x_2,r_2)\subset B_d(x_1,r_1),$ $x_1,x_2\in \re^n,$ and $0<r_1,\,r_2<\infty$. Since $\abs{B_d(x,r)}=c_d\, r^Q$ the above inequality holds true with a uniform constant $c$ and any positive $\delta\le Q.$

The growth condition  for the kernel in this context means that for
every positive constant $C_1$
there exists a positive constant $C_2$ such that for all $\vec{x},\,\vec{y},\,\vec{z}\in \re^{nm},$
\begin{eqnarray*}
\frac{{d}(\vec{x}, \vec{y})^\tau}{\prod_{i=1}^m \abs{B_{{d}}(x_i, {d}(\vec{x},\vec{y}))}}&\le& C_2\, \frac{{d}(\vec{z}, \vec{y})^\tau}{\prod_{i=1}^m\abs{B_{d}(z_i, {d}(\vec{z},\vec{y}))}}, \quad {d}(\vec{z}, \vec{y})\le C_1 \,{d}(\vec{x},\vec{y})\\
\frac{{d}(\vec{x}, \vec{y})^\tau}{\prod_{i=1}^m\abs{B_{{d}}(x_i, {d}(\vec{x},\vec{y}))}}&\le& C_2\, \frac{{d}(\vec{y}, \vec{z})^\tau}{\prod_{i=1}^m\abs{B_{{d}}(y_i,{d}(\vec{y},\vec{z}))}}, \quad {d}(\vec{y}, \vec{z})\le C_1 \,{d}(\vec{x},\vec{y}).
\end{eqnarray*}
   Both inequalities follow from the facts that $\abs{B_d(x,r)}=c_d\, r^Q$ and $\tau\le m\,Q.$
\end{proof}

\subsection{Multilinear higher-order representation formulas in Carnot groups}\label{secc:highrepformulas}

In this subsection we prove multilinear higher-order representation formulas for a family of generators of a  Carnot group (Corollaries~\ref{repformulamulti}, \ref{globalrepformulamulti}, \ref{repformulamulti3}) as consequences of their linear counterparts (Theorems {\bf A}, {\bf B}, {\bf C}).

\begin{theorA}[{\cite[p.111, Corollary E]{LW00}}]\label{repformula}
Let $\cg$ be a  Carnot group in $\rn$ of homogeneous dimension $Q,$ $k$ a positive integer, $d$ the Carnot-Carath\'eodory metric in $\rn$ with respect to a family of generators ${\bf{X}}$ of $\cg,$  $B$ a $d$-ball, and $f\in C^k(B)$. Then there exists a polynomial $P_k(B,f)$ of degree less than $k$ such that for $x\in B,$
\begin{equation*}
\abs{f(x)- P_k(B,f)(x)} \le C\, \int_B\abs{{\bf{X}}^k f(y)}\frac{d(x,y)^k}{\abs{B_d(x,d(x,y))}}\,dy+C\,\frac{r(B)^k}{\abs{B}}\int_B\abs{{\bf{X}}^kf(y)}\,dy,
\end{equation*}
where $C$ is independent of $f,$ $x$ and $B.$ Moreover, if $k\le Q$ then
\begin{equation}\label{linealrepformula}
\abs{f(x)- P_k(B,f)(x)} \le C\, \int_B\abs{{\bf{X}}^k f(y)}\frac{d(x,y)^k}{\abs{B_d(x,d(x,y))}}\,dy.
\end{equation}
\end{theorA}

\begin{corollary}[Higher-order multilinear representation formula.] \label{repformulamulti}
Let $\cg$ be a  Carnot group in $\rn$ of homogeneous dimension $Q$ and $n_1$ generators, $k$ and $m$  positive integers such that $k\le m\, Q$, $d$ the Carnot-Carath\'eodory metric in $\rn$ with respect to a family of generators ${\bf{X}}$ of $\cg,$  $B$ a $d$-ball, and $\vec{f}=(f_1,\cdots,f_m)\in (C^k(B))^m.$ Then there exists a polynomial $P_k(B,\vec{f})$ of degree less than $k$ such that for $x\in B,$
\begin{equation}\label{repformulamulti2}
\abs{\prod_{i=1}^m f_i(x)- P_k(B, \vec{f})(x)}\le \,C \sum_{\substack{\alpha_i\in \N_0^{n_1} \\ \abs{\alpha_1}+\cdots+\abs{\alpha_m}=k}} \mathcal{I}_{\cg,k}(\abs{{\bf{X}}^{\alpha_1}f_1}\chi_B, \cdots,  \abs{{\bf{X}}^{\alpha_m}f_m}\chi_B)(x),
\end{equation}
where  $C$ is independent of $\vec{f},$ $x$ and $B.$
\end{corollary}
\begin{proof}
Consider the Carnot group $\cg^{(m)}$ in $\re^{nm}$ given by the sum of $m$ copies of $\cg$ (see \cite[p. 190]{BLU07}) and note that $\cg^{(m)}$ has homogeneous dimension $m\,Q.$ Let
$\tilde{{\bf{X}}}$ be the family of generators for $\cg^{(m)}$ given by $m$ copies of ${\bf{X}}$ with appropriately added zeros and $\tilde{d}$ be the Carnot-Carath\'eodory metric in $\re^{nm}$ associated with $\tilde{{\bf{X}}}.$ Then $B_{\tilde{d}}(\vec{x},r)=\prod_{i=1}^m B_{d}(x_i,r)$ for $r>0$ and $\vec{x}=(x_1,\cdots,x_m)\in\cg^{(m)}$ (see \cite[Lemma 1]{LW98a}). Therefore $B^{m}$ is a $\tilde{d}$-ball and Theorem {\bf A} applied to $f(\vec{y})=\prod_{i=1}^m f_i(y_i),$ $\vec{y}=(y_1,\cdots,y_m),$ gives that there exists a polynomial $P_k(B^m, {f})(\vec{x})$ on $\cg^{(m)}$ of degree less than $k$ such that for all $\vec{x}=(x_1,\cdots,x_m)\in B^m$
\begin{align*}
\abs{\prod_{i=1}^m f_i(x_i)- P_k(B^m, {f})(\vec{x})}& \le C\, \int_{B^m}\abs{\tilde{{\bf{X}}}^k f(\vec{y})}\frac{\tilde{d}(\vec{x},\vec{y})^k}{\abs{B_{\tilde{d}}(\vec{x},\tilde{d}(\vec{x},\vec{y}))}}\,d\vec{y},\\
&=C\, \int_{B^m}\abs{\tilde{{\bf{X}}}^k f(\vec{y})}\frac{\tilde{d}(\vec{x},\vec{y})^k}{\prod_{i=1}^m\abs{B_d(x_i, \tilde{d}(\vec{x}, \vec{y}))}}\,d\vec{y},\\
\end{align*}
where $C$ is independent of $f,$ $\vec{x}$ and $B.$ Restricting to the diagonal, $x_1=x_2=\cdots=x_m=x\in B$,  setting $P_k(B,\vec{f})(x):=P_k(B^m,{f})(\vec{x}),$ which is a polynomial in $x$ of degree less than $k,$ and using that $\tilde{d}(\vec{x},\vec{y})\sim d(x,\vec{y})$ and  $\abs{B_d(x,r)}=c_d\, r^Q$ we obtain
\begin{align*}
&\abs{\prod_{i=1}^m f_i(x)- P_k(B, \vec{f})({x})} \le C\,  \int_{B^m}\frac{|\tilde{{\bf{X}}}^k f(\vec{y})|}{d(x, \vec{y})^{mQ-k}}\,d\vec{y}\\
& \le \, C \sum_{\substack{\alpha_i\in \N_0^{n_1} \\ \abs{\alpha_1}+\cdots+\abs{\alpha_m}=k}}\int_{B^m}\frac{\abs{{\bf{X}}^{\alpha_1}f_1(y_1)\cdots {\bf{X}}^{\alpha_m}f_m(y_m)}}{d(x, \vec{y})^{mQ-k}} d\vec{y}
\end{align*}
showing \eqref{repformulamulti2}.
\end{proof}

\bigskip

In \cite{LW04} it is proved that the polynomial in \eqref{linealrepformula} can be taken in such a way that  $\abs{B}\to \infty$ gives the following global representation formula.
\begin{theorB}[{\cite[p. 659, Theorem 3.1]{LW04}}]\label{repformulasglobal} 
Let $\cg$ be a  Carnot group in $\rn$ of homogeneous dimension $Q,$ $k$ a positive integer such that $k\le Q$,  and $d$ the Carnot-Carath\'eodory metric in $\rn$ with respect to a family of generators ${\bf{X}}$ of $\cg.$  Then there exists a constant $C$ such that
\begin{equation*}
\abs{f(x)} \le C\, \int_{\re^n}\abs{{\bf{X}}^k f(y)}\frac{d(x,y)^k}{\abs{B_d(x,d(x,y))}}\,dy, \qquad f\in C^k_c(\re^n),\,\,x\in\re^n.
\end{equation*}
\end{theorB}

Reasoning as in the proof of Corollary~\ref{repformulamulti} we obtain
\begin{corollary}[Higher-order global multilinear representation formula.] \label{globalrepformulamulti}
Let $\cg$ be a  Carnot group in $\rn$ of homogeneous dimension $Q$ and $n_1$ generators, $k$ and $m$  positive integers such that $k\le m\, Q$,  and $d$ the Carnot-Carath\'eodory metric in $\rn$ with respect to a family of generators ${\bf{X}}$ of $\cg.$  There exists a constant $C$ such that
\begin{equation}\label{repformulamulti2}
\abs{\prod_{i=1}^m f_i(x)}\le \,C \sum_{\substack{\alpha_i\in \N_0^{n_1} \\ \abs{\alpha_1}+\cdots+\abs{\alpha_m}=k}} \mathcal{I}_{\cg,k}(\abs{{\bf{X}}^{\alpha_1}f_1}, \cdots,  \abs{{\bf{X}}^{\alpha_m}f_m})(x),\quad f_i\in C_c^k(\re^n), \,\, x\in \re^n.
\end{equation}
\end{corollary}

\bigskip

The following representation formula is well known (see, for instance, \cite[p.236]{BLU07}).

\begin{theorC} \label{repformula2} Let $\mathcal{L}$ be  a sub-Laplacian on a  Carnot group $\cg$ in $\re^n$ of  homogeneous dimension strictly larger than 2. If $\Gamma$ is the fundamental solution for $\mathcal{L}$ then
\[\phi(x)=-\int_{\re^n}\Gamma(x^{-1}\diamond y)\,\mathcal{L}\phi(y)\,dy, \qquad \phi\in\mathcal{C}^\infty_0(\re^n),\,x\in\re^n.\]
\end{theorC}

\begin{remark}  The homogeneous dimension of $\cg$ being  strictly larger than 2 guaranties the existence of a fundamental solution for $\mathcal{L},$ which is unique. Also  $\Gamma(x)\sim d(x,0)^{2-Q},$ $x\ne 0, $ where $d$ is the Carnot-Carath\'eodory  metric in $\re^n$ associated with the family of generators corresponding to the sub-Laplacian $\mathcal{L}.$
\end{remark}

\begin{corollary}[Second-order global multilinear  representation formula with a sub-Laplacian.] \label{repformulamulti3}
Let $\cg$ be a  Carnot group in $\rn$ of homogeneous dimension $Q,$  $m$ a positive integer such that $m\, Q>2$, ${\bf{X}}$ a family of generators of $\cg$ and  $\mathcal{L}$ its sub-Laplacian. There exists a constant $C$ such that 
\begin{equation}\label{repformulamulti4}
\abs{\prod_{i=1}^m f_i(x)}\le C\, \sum_{i=1}^m \mathcal{I}_{\cg,2}(\abs{f_1}, \cdots, \abs{\mathcal{L} f_i}, \cdots, \abs{f_m})(x), 
\end{equation}
for all $x\in\re^n$ and $f_i\in \mathcal{C}^\infty_c(\re^n),$ $i=1,\cdots,m.$ 
\end{corollary}
\begin{proof}  Let  $d$ be the Carnot-Carath\'eodory metric in $\rn$ with respect to ${\bf{X}}$ and let $\cg^{(m)},$ $\tilde{{\bf{X}}}$ and $\tilde{d}$ be as in the proof of Corollary~\ref{repformulamulti}. If $\tilde{\Gamma}$ is the fundamental solution for the sub-Laplacian $\tilde{\mathcal{L}}$ corresponding to $\tilde{{\bf{X}}}$ then  Theorem {\bf B} gives
\[f_1(x_1)f_2(x_2)\cdots f_m(x_m)=-\int_{\re^{mn}}\tilde{\Gamma}((\vec{x})^{-1}\tilde{\diamond} \vec{y})\,\tilde{\mathcal{L}}(f_1\cdots f_m)(\vec{y})\,d\vec{y}, \]
for  $\vec{x}=(x_1,\cdots,x_m)\in\re^{mn},$  $\tilde{\diamond}$ the group operation in $\cg^{(m)}$ and $f_i\in \mathcal{C}^\infty_c(\re^n),$ $i=1,\cdots,m.$ 

We note that $\tilde{\Gamma}((\vec{x})^{-1}\diamond \vec{y})\sim  \tilde{d}(\vec{x},\vec{y})^{2-mQ}\sim d(\vec{x},\vec{y})^{2-mQ}$  and that $\tilde{\mathcal{L}}(f_1\cdots f_m)(\vec{y})=\sum_{i=1}^m \mathcal{L}f_i (y_i)\cdot \Pi_{j\not=i} f_j(y_j)$ for $\vec{y}=(y_1,\cdots,y_m)\in\re^{mn}.$ By these remarks and by taking $x:=x_1=x_2=\cdots=x_m$ in the above formula we obtain \eqref{repformulamulti4}.

\end{proof}

\section{An application to weighted Leibniz-type rules in Campanato-Morrey spaces}\label{secc:leib}

 Let $\cg$ be a  Carnot group in $\rr^n$, $d$ the Carnot-Carath\'eodory metric in $\rr^n$ with respect to a family of generators  ${\bf{X}}$ of $\cg$, $w \geq 0$ a weight and   $p, \lambda > 0$. A function  $f \in L^1_{loc}(\rn, w^p)$ is said to belong to the weighted
Morrey space $L^{p,\lambda}(w)$ if
$$
\norm{f}{L^{p,\lambda}(w)} = \sup_{B} \left(\frac{1}{|B|^{\lambda/Q}} \int_B |f(x) w(x)|^p\, dx \right)^{1/p} < \infty
$$
where the supremum is over all $d$-balls $B$.  Next, we define the weighted Campanato space of order $k$, $\mathcal{L}^{p,\lambda}_k(w)$.  Let $\poly_k$ be the collection of polynomials in $\cg$ of degree less than $k$.   We write $f\in \Lcm_k^{p,\lambda}(w)$ if
$$
\norm{f}{\mathcal{L}_k^{p,\lambda}(w)} = \sup_{B} \inf_{P \in \poly_k} \left( \frac{1}{|B|^{\lambda/Q}}\int_B \left(|f(x) - P(x)| w(x)\right)^p \, dx\right)^{1/p}  < \infty.
$$
As a  consequence of Theorem \ref{high2wthm}   we have the following.
\begin{theor} \label{LeibnizCampanato} Let $1<p_1, p_2<\infty$, $p$ be defined by
 $ \frac{1}{p}=\frac{1}{p_1}+\frac{1}{p_2},$ $q\geq p$ and $\lambda,\lambda_1,\lambda_2\in (0,\infty)$ be such that $\frac{\lambda}{q}=\frac{\lambda_1}{p_1}+\frac{\lambda_2}{p_2}$.
If $\cg$ is a  Carnot group in $\rn$ of homogeneous dimension $Q$ and $n_1$ generators, $k$ is a positive integer such that $k\le 2\, Q$, $d$ is the Carnot-Carath\'eodory metric in $\rn$ with respect to a family of generators ${\bf{X}}$ of $\cg$ and $u, \, v_1, \, v_2$ are weights on $\re^n$
satisfying condition \eqref{2wq>1} if $q>1$ or condition
\eqref{2wq<1} if $q\le 1$ with $m=2$, then
$$\|fg\|_{\Lcm_k^{q,\lambda}(u)}\leq \,C \sum_{\substack{\alpha_i\in \N_0^{n_1} \\ \abs{\alpha_1}+\abs{\alpha_2}=k}}
\left\| {{\bf{X}}^{\alpha_1}} f \right\|_{L^{p_1,\lambda_1}(v_1)}\|{\bf{X}}^{\alpha_2}g\|_{L^{p_2,\lambda_2}(v_2)}.$$

\end{theor}

\end{document}